\documentclass[10pt,reqno]{amsart}

\usepackage{amssymb, verbatim, enumerate}
\usepackage{amsmath}
\usepackage{graphicx}
\usepackage{epstopdf,fancybox,color}

\newtheorem{theorem}{Theorem}

\newtheorem{lemma}[theorem]{Lemma}

\theoremstyle{definition}
\newtheorem{definition}[theorem]{Definition}

\theoremstyle{remark}
\newtheorem{remark}{Remark}

\newcommand{\C}{{\mathbb{C}}}

\title[Dissipative eigenvalue problems for a singular quantum Dirac system]{Dissipative eigenvalue problems for a singular quantum Dirac system}

\author[B.P. Allahverdiev]{Bilender P. Allahverdiev}
\address[B.P. Allahverdiev]{Department of Mathematics,
              Khazar University, AZ1096 Baku, Azerbaijan and Research Center of Econophysics, UNEC-Azerbaijan State University of Economics, Baku, Azerbaijan }
\email[B.P. Allahverdiev]{bilenderpasaoglu@gmail.com}

\author[Y. Aygar]{Yelda Aygar*}
\address[Y. Aygar*]{Department of Mathematics,
              Faculty of Science,
              Ankara University,
              06100 Tando\u{g}an, Ankara, Turkey}
\email[Y. Aygar (Corresponding author)]{yaygar@ankara.edu.tr} 
\date{\today}


\date{\today}

\begin{document}

\begin{abstract}
In this paper, dissipative singular $q$-Dirac operators are examined in Hilbert space $\mathcal{L}_{q}^{2}(q^{{\mathbb{Z}}};\mathbb{C}^{2})$, where they arise as extensions of a minimal symmetric operator in the limit-point case. We develop a selfadjoint dilation and get its incoming and outgoing spectral representations, enabling the explicit computation of the scattering matrix associated with the dilation. Furthermore, we construct a functional model for the dissipative operator and express its characteristic function in terms of the Titchmarsh-Weyl function of the corresponding selfadjoint operator. Finally, we establish results concerning the completeness of the system of eigenvectors and associated vectors for these dissipative Dirac operators.
\end{abstract}

\maketitle

\section{Introduction}\label{Sec:1}

Quantum calculus ($q$-calculus)  i.e; calculus without limits is a generalization of classical calculus where the notation of a derivative and integral is defined without taking limits. The development of $q$-calculus started in the 1740s by Euler. The progress of $q$-calculus continued under C.F. Gauss, who invented the hypergeometric series and their continuity relations and Jakson developed the $q$-derivative and $q$-integral formally. Since then, many authors studied kinds of topics in quantum calculus. Quantum calculus has a wide range of applications across mathematical physics, engineering and computational sciences. For more information on the basic definitions and theorems, we refer the reader to the books \cite{10,88,12,16}. On the other hand, $q$-difference equations become a $q$-analogue of differential equations. In this paper, we consider the $q$ analog of the Dirac system given in \cite{3}. Then, we investigate the spectral properties of the operators related to this system by using the operator theory. As you know, dissipative operators is one of the important class of operators and there are some basic methods for studying the spectral analysis of dissipative operators such as resolvent analysis, Riesz integrals and theory of dilations with applications of functional models. Note that, the development of functional models for dissipative operators, which serve as natural analogues of the spectral decompositions for selfadjoint operators, is carried out using the Sz. Nagy-Foias dilation theory \cite{20} and Lax-Phillips scattering theory \cite{18}. The spectral properties of dissipative operators have been examined in \cite{1,2,3,21,22,23} through the application of their functional model.
Firstly, we recall some necessary concepts of quantum calculus and secondly, we introduce the $q$-Dirac system that we will investigate.
Let $q$ be a positive number with $0<q<1$, $B\subset{\mathbb{R}}:=(-\infty,\infty)$. A $q$-difference equation is known as an equation that contains $q$-derivatives of a function defined on $B$. Let $f$ be a complex-valued function on $B$, then the $q$-difference operator $D_{q}$ is defined as 
\begin{equation*}
D_{q}f(t):=\frac{f(qt)-f(t)}{t(q-1)}, \quad t\in B\backslash \{0\}.
\end{equation*}
If $0\in B$, the $q$-derivative of a function $f$ at zero is defined by $(0<q<1)$
\begin{equation*}
D_{q}f(0):=\underset{n\rightarrow \infty}{\lim}\frac{f(q^{n}t)-f(0)}{q^{n}%
t},\quad \text{ }t\in B,
\end{equation*}
whenever the limit exists and does not depend on $t$. Assuming the existence of $D_{q}f(0)$, the formulation of the extension problems requires that definition $D_{q^{-1}}$ be established in a manner analogous to $D_{q}$ as 
\begin{equation*}
D_{q^{-1}}y(t):=\left \{
\begin{array}
[c]{c}%
\frac{y(t)-y(q^{-1}t)}{t(1-q^{-1})},\quad \text{ }t\in B\backslash \{0\},\\
\\
D_{q}y(0),\quad \text{ }t=0.
\end{array}
\right.
\end{equation*}
Furthermore, Jakson gives the converse of the $q$-difference operator as Jakson's $q$-integration in \cite{15} by 
\begin{equation*}
\int_{0}^{t}f(\xi)d_{q}\xi:=t(1-q)\sum_{n=0}^{\infty}q^{n}f(q^{n}t),\quad \text{ }t\in B,
\end{equation*}
if the series converges, and
\begin{equation*}
\int_{a}^{b}f(t)d_{q}t:=\int_{0}^{b}f(t)d_{q}t-\int_{0}^{a}f(t)d_{q}t,\text{
}\quad a,b\in B.
\end{equation*}
Note that the following equations are obtained by using the definitions.
\[
D_{q^{-1}}y(t)=(D_{q}y)(q^{-1}t),\text{ }D_{q}^{2}y(q^{-1}t)=qD_{q}%
[D_{q}y(q^{-1}t)]=D_{q^{-1}}D_{q}y(t).\]
In the remainder of the paper, we deal only with functions $q$-regular at zero, that is, functions satisfying $\underset{n\rightarrow \infty}{\lim}f(q^{n}t)=f(0)$ 
for all $t\in B$. It is known from the literature that the class of the functions which are $q$-regular at zero consists in continuous functions. If $f$ and $g$ are $q$-regular at zero, we can give a rule of $q$-integration by parts given by 
\[
\int_{0}^{a}g(t)D_{q}f(t)d_{q}t=(fg)(a)-(fg)(0)-\int_{0}^{a}D_{q}%
g(t)f(qt)d_{q}t.
\]
Now, let us give some information about $q$-Dirac system. Firstly, let us recall one dimensional Dirac system
\begin{align*}
- y_2' + p(x)y_1 &= \lambda y_1 \\
\phantom{-} y_1' + r(x)y_2 &= \lambda y_2,
\end{align*}
where $\lambda$ is a complex parameter, $p$ and $r$ are $q$-regular at zero. This system describes a relativistic electron in the electrostatic field (see\cite{24}). If we replace the ordinary derivatives by the $q$-derivative, we obtain the following $q$-Dirac system
\begin{align*}
-\frac{1}{q} D_{q^{-1}}y_2 + p(x)y_1 &= \lambda y_1 \\
\phantom{-}D_{q} y_1 + r(x)y_2 &= \lambda y_2,
\end{align*}
where $q$ is a positive number and is less than $1$. In \cite{7}, the authors investigate the symmetric $q$-Dirac operator. They describe dissipative, accumulative selfadjoint and the other extensions of such operators with general boundary conditions. Furthermore, they get a selfadjoint dialtion of dissipative operator and determine the scattering matrix of dilation. On the other hand in \cite{8}, the authors construct Titchmarsh-Weyl Theory for singular $q$-Dirac systems and in \cite{98}, the authors establish a Parseval equality and expansion formula in eigenfunctions for the $q$-Dirac operator on the whole line. Unlike these, in this paper, we will investigate dissipative eigenvalue problems for singular $q$-Dirac system. Let $\mathcal{L}_{q}^{2}(q^{{\mathbb{Z}}};\mathbb{C}^{2})$ (where $q^{{\mathbb{Z}}}:=\left \{  q^{n}:n\in{\mathbb{Z}}\right \}$ and 
${\mathbb{Z}}:=\{0,\pm1,\pm2,\pm3,...\}$) denote the Hilbert space consisting of all complex-vector functions defined on $q^{{\mathbb{Z}}}$ satisfying 
\begin{equation*}
\left(\int_{0}^{\infty}\left\| f(t) \right\|_{\mathbb{C}^{2}}\;d_{q}t\right)^\frac{1}{2}<\infty
\end{equation*}
and with the inner product 
\begin{equation*}
(f,g)=\int_{0}^{\infty}\left(f(t),\overline{g(t)}\right)_{\mathbb{C}^{2}}\;d_{q}t,\quad f,g\in \mathcal{L}_{q}^{2}(q^{{\mathbb{Z}}};\mathbb{C}^{2}).
\end{equation*}
In this paper, we investigate the maximal dissipative singular $q$-Dirac operators acting in $\mathcal{L}_{q}^{2}(q^{{\mathbb{Z}}};\mathbb{C}^{2})$, that the extensions of a minimal symmetric
 operator in Weyl's limit point case (with deficiency indices (1,1)). Firstly, we establish a selfadjoint dilation and its incoming and outgoing spectral representations, which satisfies determining the scattering matrix of dilation according to the Lax and Phillips method \cite{18}. Secondly, we give a functional model for the dissipative operator  and determine its characteristic function in terms of the Titchmarcsh-Weyl function of selfadjoint operator by using the spectral representations. Finally, we prove theorems on completeness of the system of eigenvectors and associated vectors of the dissipative $q$-Dirac operators with the help of the results obtained for characteristic functions.
\section{Selfadjoint dilation of the dissipative operator}\label{Sec:2}

Let us introduce the $q$-analogue of the one-dimensional Dirac system as
\begin{equation}\label{E1}
\left(
\begin{array}{cc}
0 & -\frac{1}{q}{D_{q}}^{-1} \\
D_{q} & 0
\end{array}
\right)
\begin{pmatrix}
y_1 \\
y_2
\end{pmatrix}
+
\left(
\begin{array}{cc}
p(t) & 0 \\
0 & r(t)
\end{array}
\right)
\begin{pmatrix}
y_1 \\
y_2
\end{pmatrix}
=
\lambda
\begin{pmatrix}
y_1 \\
y_2
\end{pmatrix},
\end{equation}
where $\lambda$ is a complex spectral parameter, $p$ and $r$ are real valued functions defined on $q^{\mathbb{Z}}$ and continuous at zero (or $q$-regular at zero) such that $p(t)\neq 0$ and $r(t)\neq 0$ for all $t\in q^{\mathbb{Z}}$. Furthermore; $D_{q}$ is the $q$-difference operator given in previous section and $q$ is a positive number which is less than $1$. Note that, we can write the equation \eqref{E1} by using the expression 
\[
(l_{1}y)(t) := 
\begin{cases}
-\dfrac{1}{q} D_q^{-1} y_2(t) + p(t) y_1(t) \\
\phantom{-}D_q y_1(t) + r(t) y_2(t)
\end{cases}
\]
as $(l_{1}y)(t)=\lambda y(t)$, $t\in q^{\mathbb{Z}}$.
Now, we introduce a convenient Hilbert space $H:=\mathcal{L}_{q}^{2}(q^{\mathbb{Z}};\mathbb{C}^{2})$ of vector-valued functions using the inner product
\begin{equation*}
(f,g):=\int_{0}^{\infty}\left(f(t),g(t)\right)_{\mathbb{C}^{2}}d_{q}t.
\end{equation*}
Assume that $D_{max}$ is a linear set consisting of all vector-valued functions $y=\begin{pmatrix}
y_1(t) \\
y_2(t)
\end{pmatrix}\in \mathcal{L}_{q}^{2}(q^{\mathbb{Z}};\mathbb{C}^{2})=H$ such that $y_1$ and $y_2$ are $q$-regular at zero and $l_{1}y\in H$. We define the maximal operator $L_{max}$ on $D_{max}$ by $L_{max}y={l_{1}}y$. For arbitrary vectors $y,z\in D_{max}$, we write Green's formula \cite{7,8} by
\begin{eqnarray} \label{E2}
\int_{0}^{t}\left(({l_{1}}y)(s),z(s)\right)_{\mathbb{C}^{2}}\;d_{q}s-\int_{0}^{t}(y(s),\left({l_{1}}z)(s)\right)_{\mathbb{C}^{2}}\;d_{q}s
&=&\left( {l_{1}}y,z\right)-\left(  y,{l_{1}}z\right) =\left[y,z\right]  (t)-\left[  y,z\right]  (0), 
\end{eqnarray}
for $t\in q^{\mathbb{Z}}$, where $\left[y,z\right]$ is defined as follows:
\[
\left[  y,z\right]  (t):=W\left[  y,\overline z\right](t)=y_{1}(t)\overline{z_{2}(q^{-1}t)}-\overline {z_{1}(t)}y_{2}(q^{-1}t).\]
It is evident from \eqref{E2} that the limit 
$$\left[  y,z\right]  (\infty):=\lim_{n\rightarrow \infty}\left[  y,z\right] (q^{-n})$$
exists and is finite for $y,z\in D_{max}$. On the other hand, we know that $y_{1}$ and $y_{2}$ are $q$-regular at zero. It presents us that $\lim_{n\rightarrow \infty}  y_{1}(q^{n})$ and $\lim_{n\rightarrow \infty}  y_{2}(q^{n})$ exist and are finite. As a result of this for any vector function $y\in D_{max}$, we can define $y(0)=\begin{pmatrix}
y_1(0) \\
y_2(0)
\end{pmatrix}$ by using these limits as
\begin{equation}\label{E3}
y_1(0):=\lim_{n\rightarrow \infty}  y_{1}(q^{n})\quad {\text{and}}\quad  y_2(0):=\lim_{n\rightarrow \infty}  y_{2}(q^{n}).
\end{equation}
In $H$, we consider another dense linear set $D_{min}$ consisting of finite nonzero vector-valued functions. If we denote the restriction of the operator $L_{max}$ to $D_{min}$ by $L_{0}$, then it becomes closed and symmetric from \eqref{E2} and \eqref{E3}. If we called the closure of $L_{0}$ by $L_{min}$, then the domain of $L_{min}$ consists of precisely those vectors $y\in D_{max}$ satisfying the condition 
\begin{equation}\label{E4}
y_1(0)= y_{2}(0)=0 \quad {\text{and}}\quad   \left[y,z\right](\infty)=0
\end{equation}
for arbitrary $y=\begin{pmatrix}
y_1 \\
y_2
\end{pmatrix}\in D_{max}$.
It follows from that the operator $L_{min}$ is a closed symmetric operator with defect indices $(1,1)$ or $(2,2)$, 
and satisfies $L^{*}_{min}=L_{max}$ \cite{111,13,14,24}. The operators $L_{min}$ and $L_{max}$ are called the minimal and maximal operators, respectively. Since the defect indices of $L_{min}$ $(1,1)$ or $(2,2)$ the case of limit-point occurs for $l_{1}(y)$. It is also known from \cite{8}[section5] that for the $q$-Dirac system, only the limit point case exists, the limit circle case does not occur. Now, to take one step further, let us recall the following definition.
\begin{definition}
A linear operator $A$ (with dense domain $D(A)$) acting on any Hilbert space $H$ is called dissipative (accumulative) if $\operatorname{Im}%
(Af,f) $ $\geq0$ for all $f\in D(A)$ ($\operatorname{Im}(Af,f)\leq0$) for all $f\in D(A)$ and maximal dissipative (accretive) if it does not have a proper dissipative (accretive) extension.
\end{definition}
Consider the operator $T_{h}$ with domain $D(T_{h})$ consisting of vectors $y\in D_{max}$ which satisfy the following conditions
\begin{equation}\label{E5}
y_2(0)-hy_{1}(0)=0, \quad \operatorname{Im}h>0.
\end{equation}
\begin{theorem}\label{T1}
The operator $T_{h}$ is dissipative in $\mathcal{L}_{q}^{2}(q^{\mathbb{Z}};\mathbb{C}^{2})$.
\end{theorem}
\begin{proof}
Assume that $y=\begin{pmatrix}
y_1(t) \\
y_2(t)
\end{pmatrix}\in D(T_{h})$. Then, we write
\begin{equation}\label{E6}
(T_{h}y,y)-(y,T_{h}y)=\left[y,y\right](\infty)-\left[y,y\right](0).
\end{equation}
We know that limit-point case occurs at $\infty$, so it gives $\left[y,y\right](\infty)=0$. By using \eqref{E5}, we get\begin{equation}\label{E7}
\left[y,y\right](0)=-2\operatorname{Im}h\left| y_{1}(0) \right|^{2}.
\end{equation}
Substituting \eqref{E7} in \eqref{E6}, we write $\operatorname{Im}(T_{h}y,y)=\operatorname{Im}h\left| y_{1}(0) \right|^{2}$, and since $\operatorname{Im}h>0$, it gives $\operatorname{Im}(T_{h}y,y)>0$, it means that $T_{h}$ is dissipative in $\mathcal{L}_{q}^{2}(q^{\mathbb{Z}};\mathbb{C}^{2})$.\end{proof} 
It is clear from \eqref{E7} that if $\operatorname{Im}h<0$ ($\operatorname{Im}h=0$ or $h=\infty$), then $T_{h}$ becomes an accumulative (self-adjoint) operator in $\mathcal{L}_{q}^{2}(q^{\mathbb{Z}};\mathbb{C}^{2})$. Here for $h=\infty$, condition \eqref{E5} should be replaced by $y_{1}(0)=0$.
Next, we will introduce the Hilbert spaces $\mathcal{L}^{2}(\mathbb{R_{-}}):=(-\infty,0]$ and $\mathcal{L}^{2}(\mathbb{R_{+}}):=[0,\infty)$ which consist all complex-valued functions satisfying \[
\int_{-\infty}^{0}\left \vert y(t)\right \vert ^{2}dt<\infty.
\]
and \[
\int_{0}^{\infty}\left \vert y(t)\right \vert ^{2}dt<\infty,
\]
respectively, with the following inner products \begin{equation*}
(y,z)=\int_{-\infty
}^{0}y(t)\overline{z(t)}dt,  \text{ } 
(y,z)=\int_{0
}^{\infty}y(t)\overline{z(t)}dt.\text{ }
\end{equation*}%
Let call the Hilbert space 
$\mathcal{H}=\mathcal{L}%
^{2}(\mathbb{R}_{-})\oplus \mathcal{L}_{q}^{2}(q^{\mathbb{Z}};\mathbb{C}^{2})\oplus \mathcal{L}^{2}(\mathbb{R}_{+})$, with the name Hilbert space of dilation. Furthermore, we call the Sobolev space by $\mathcal{W}_{2}^{1}\left( \mathbb{R}_{\mp
}\right)$, which consists in all functions $f\in\left( \mathbb{R}_{\mp
}\right)$ such that $f$ are locally absolutely continuous functions on $\left( \mathbb{R}_{\mp
}\right)$ and $%
f^{\prime }\in \mathcal{L}^{2}(\mathbb{R}_{\pm })$. Assume that $D(M_{h})$ is the set of all vectors $f=\langle \phi _{-},y,\phi _{+}\rangle$ in $\mathcal{H}$, where $\phi _{-}\in \mathcal{W}_{2}^{1}\left( \mathbb{R}_{-}\right)$, $\phi _{+}\in \mathcal{W}_{2}^{1}\left( \mathbb{R}_{-}\right)$, and $y\in D(M_{h})$ satisfying the conditions
\begin{equation}\label{E8}
y_{2}(0)-hy_{1}(0)=\delta \phi _{-}(0),  \quad  
y_{2}(0)-\overline{h}y_{1}(0)=\delta \phi _{+}(0),
\end{equation}
where $\delta^{2}:=2\operatorname{Im}h,\quad \delta>0$. If we define 
\begin{equation}\label{E9}
M\langle \phi _{-},y,\phi _{+}\rangle=\langle i\frac{d\phi _{-}}{d\xi},{l_{1}},i \frac{d\phi _{+}}{d\tau}\rangle
\end{equation}
with $M_{h}f=Mf$ for $f\in D(M_{h})$, then we can give the following theorem.
\begin{theorem}\label{T2}
The operator $M_{h}$ is selfadjoint in $\mathcal{H}$.
\end{theorem}
\begin{proof}
Let $f=\langle \phi _{-},y,\phi _{+}\rangle, g=\langle \psi _{-},z,\psi _{+}\rangle\in D(M_{h})$ such that $y=\begin{pmatrix}
y_1(t) \\
y_2(t)
\end{pmatrix}$ and $z=\begin{pmatrix}
z_1(t) \\
z_2(t)
\end{pmatrix} $.
\end{proof}
Then with a direct calculation, we write
\begin{eqnarray}\label{E10}
(M_{h}f,g)_{\mathcal{H}}=\int_{-\infty
}^{0}i{\phi _{-}}^{\prime}\overline{\psi _{-}}d\xi+\int_{0
}^{\infty}i{\phi _{+}}^{\prime}\overline{\psi _{+}}d\tau  =i\phi _{-}(0)\overline{\psi _{-}(0)}-i\phi _{+}(0)\overline{\psi _{+}(0)}-[y,z](0)+(f,M_{h}g)_{\mathcal{H}}. 
\end{eqnarray}
By using \eqref{E8}, it follows $(M_{h}f,g)_{\mathcal{H}}=(f,M_{h}g)_{\mathcal{H}}$, i.e., $M_{h}$ is symmetric. It can be seen that $M_{h}$ and ${M_{h}}^{*}$ are generated by the same differential expression \eqref{E9}. It is clear that the equality \eqref{E10} can be rewritten as
\begin{equation}\label{E11}
i\phi _{-}(0)\overline{\psi _{-}(0)}-i\phi _{+}(0)\overline{\psi _{+}(0)}-[y,z](0)=0.
\end{equation}
On the other side, we write
\begin{equation}\label{E12}
y_{1}(0)=-\frac{i}{\delta}\left(\phi _{+}(0)-\phi _{-}(0)\right) \quad y_{2}(0)=\delta \phi _{-}(0)-\frac{ih}{\delta}\left(\phi _{+}(0)-\phi _{-}(0)\right)
\end{equation}
by using the boundary condition \eqref{E8} for $y_{1}(0)$ and $y_{2}(0)$. Substituting \eqref{E12} in \eqref{E11}, we obtain 
\begin{equation}\label{E13}
i\phi _{-}(0)\overline{\psi _{-}(0)}-i\phi _{+}(0)\overline{\psi _{+}(0)}=[y,z](0)
=-\frac{i}{\delta}\left[\phi _{+}(0)-\phi _{-}(0)\right]\overline{y_{2}(0)}-\delta\left[\phi _{-}(0)-\frac{ih}{{\delta}^2}\left(\phi _{+}(0)-\phi _{-}(0)\right)\right]\overline{z_{1}(0)}
\end{equation}
Since the values $\phi_{\pm}(0)$ can be arbitrary complex numbers, a comparison of the coefficient of $\phi_{\pm}(0)$ on the left and right of \eqref{E13} shows that the vector $g=\langle\psi _{-},z,\psi _{+}\rangle$ satisfies the following boundary conditions:
$$z_{2}(0)-hz_{1}(0)=\delta\psi _{-}(0)$$ and $$z_{2}(0)-\overline{h}z_{1}(0)=\delta\psi _{+}(0).$$ This implies that $M_{h}^{*}\subseteq M_{h}$, and hence $M_{h}^{*}=M_{h}$. So, $M_{h}$ is selfdjoint in $\mathcal{H}$. \newline
To present next Theorem, let us give a short expression. The selfadjoint operator $M_{h}$ generates a unitary group in $\mathcal{H}$ as $U_{h}(s)=\exp[iM_{h}s],\quad -\infty<s<\infty$. If we consider the mappings acting on $\mathcal{H}$ and $\mathcal{L}_{q}^{2}(q^{\mathbb{Z}};\mathbb{C}^{2})$ by $P_{1}$ and $P_{2}$ with the rules $P_{1}:\mathcal{H}\to\mathcal{L}_{q}^{2}(q^{\mathbb{Z}};\mathbb{C}^{2}),\quad 
\langle \phi _{-}, y, \phi _{+}\rangle \to y $
and $P_2 : L^2(\mathbb{R}^2) \to \mathcal{H}, \quad 
y \to \langle 0, y, 0 \rangle$, respectively. It is known that the operator family $Z_{h}(s)=P_{1}U_{h}(s)P_2$,  $s\ge0$ is strongly continuous semigroup of completely nonunitary contraction on $\mathcal{L}_{q}^{2}(q^{\mathbb{Z}};\mathbb{C}^{2})$. Let us consider the generator $B_{h}y=lim_{s\to0^{+}}\left[(is)^{-1}(Z_{h}(s)y-y)\right]$, where the domain of $B_{h}$ consists of all the vectors for which this limit exsists. The operator $B_{h}$ is a maximal dissipative operator, and selfadjoint dilation of $B_{h}$ is given as $M_{h}$.
In the next Theorem, we will show that $B_{h}=T_{h}$.
\begin{theorem}\label{T3}
The operator $M_{h}$ is a selfadjoint dilation of $T_{h}$.
\end{theorem}
\begin{proof}
To get the proof of theorem, let us show $B_{h}=T_{h}$. Assume that
\begin{equation}\label{E14}
(M_{h}-\lambda I)^{-1}P_{2}x=g=\langle \psi _{-},y,\psi _{+}\rangle,
\end{equation}
where $x,y\in \mathcal{L}_{q}^{2}(q^{\mathbb{Z}};\mathbb{C}^{2})$ and $\operatorname{Im}\lambda<0$. Then, we write $(M_{h}-\lambda I)g=P_{2}x$ and the equation \eqref{E11} is also equivalent to $l_{1}y-\lambda y=x$ and $\psi _{-}(\xi)=\psi _{-}(0)e^{-i\lambda\xi},$\quad $\psi _{+}(\tau)=\psi _{+}(0)e^{-i\lambda\tau}.$ Since $\psi _{-}$ belongs to $\mathcal{L}^{2}(\mathbb{R_{-}})$, then $\psi _{-}(0)=0,$ and consequently $y$ satisfies the boundary condition $y_{2}(0)-hy_{1}(0)=0.$ Therefore, $y\in D(T_{h}),$ and since a point $\lambda$ with $\operatorname{Im}\lambda<0$ can not be an eigenvalue of a dissipative operator $T_{h},$ it follows that $y=(T_{h}-\lambda I)^{-1}x.$ Further, $\psi _{+}(0)=\delta^{-1} \left (y_{2}(0)-\overline{h}y_{1}(0)\right).$ Thus \eqref{E14} is equivalent to the following equation
\begin{equation*}
    (M_{h}-\lambda I)^{-1}P_{2}x=\langle 0,(T_{h}-\lambda I)^{-1}x,\delta_{-1}\left (y_{2}(0)-\overline{h}y_{1}(0)\right)e^{-i\lambda\xi}\rangle
\end{equation*}
for $x\mathcal{L}_{q}^{2}(q^{\mathbb{Z}};\mathbb{C}^{2})$ and $\operatorname{Im}\lambda<0$. Then applying the mapping $P_{1}$ to the last equality, we find 
\begin{equation}\label{E15}
   P_{1}(M_{h}-\lambda I)^{-1}P_{2}x= (T_{h}-\lambda I)^{-1}x.
\end{equation}
On the other side, the following equality satisfies for $\operatorname{Im}\lambda<0$:
\begin{eqnarray}\nonumber
   (T_{h}-\lambda I)^{-1}&=&P_{1}(M_{h}-\lambda I)^{-1}P_{2}\\\label{E16}
  &=&-iP_{1}\int_{0
}^{\infty}U_{h}(s)e^{-i\lambda s}dsP_{2}\\ \nonumber
&=&-i\int_{0
}^{\infty}Z_{h}(s)e^{-i\lambda s}ds\\ \nonumber
&=&(B_{h}-\lambda I)^{-1}.
\end{eqnarray}
Using \eqref{E15} and \eqref{E16}, we get the proof.
\end{proof}
\section{Scattering theory of the dilation, functional model of the dissipative operators}\label{Sec:3}
In this section, we present some scattering properties of the dilation of the dissipative operators. Also, we create a functional model for dissipative operators and express its characteristic function in terms of the Titchmarch-Weyl function of the corresponding selfadjoint operator. Let us consider two solutions of the equation $(l_{1}y)(t)=\lambda y(t)=\begin{pmatrix}
\lambda y_1(t) \\
\lambda y_2(t)
\end{pmatrix},$ \quad $t\in q^{\mathbb{Z}}$ as $\phi(t,\lambda)=\begin{pmatrix}
 \phi_1(t,\lambda) \\
\phi_2(t,\lambda)\end{pmatrix}$ and $\psi(t,\lambda)=\begin{pmatrix}
 \psi_1(t,\lambda) \\
\psi_2(t,\lambda)\end{pmatrix}$ that satisfy the conditions
\begin{equation}\label{3.1}
  \phi_1(0,\lambda)=0,\quad  \phi_2(0,\lambda)=1, \quad \psi_1(0,\lambda)=1,\quad  \psi_2(0,\lambda)=0.
\end{equation}
The Titchmarsh-Weyl function $m_{\infty}$ of the selfadjoint operator $T_{\infty}$ generated by the boundary condition $y_{1}(0)=0$ is uniquely determined from the condition 
\begin{equation*}
    \psi(t,\lambda)+m_{\infty}(\lambda)\phi(t,\lambda)\in \mathcal{L}_{q}^{2}(q^{\mathbb{Z}};\mathbb{C}^{2}),\quad \operatorname{Im}\lambda\neq 0
\end{equation*}
\cite{10,11,19,20}. It is known from literature that, in general Weyl-Titchmarsh function $m_{\infty}$ is not a meromorphic function on the complex plane, but it is holomorphic function with $\operatorname{Im}\lambda\neq \textit{}0,\;\operatorname{Im}\lambda \operatorname{Im}m_{\infty}(\lambda)>0$ and $m_{\infty}$ satisfies $\overline{m_{\infty}(\lambda)}=m_{\infty}(\overline{\lambda})$ whenever $\operatorname{Im}\lambda\neq 0$ \cite{10,11}. Further, assume that the function $m_{\infty}$ is meromorphic in $\mathbb{C}$. Then it gives that any selfadjoint extension of the operator $L_{min}$ has a purely discrete spectrum (see \cite{13,14,19,24}). Due to an important property of the unitary group $U_{h}(s)=exp[iM_{h}s]$, $-\infty<s<\infty$, we can apply the Lax-Phillips scheme \cite{18} to it. To state more clearly, it has incoming and outgoing subspaces $\mathcal{D}_{-}:=\left \langle \mathcal{L}^{2}(\mathbb{R}%
_{-}),0,0\right \rangle $ and $\mathcal{D}_{+}:=\left \langle 0,0,\mathcal{L}%
^{2}(\mathbb{R}_{+})\right \rangle$ satisfying the following items:
\begin{itemize}
\item [1)] $U_{h}(s)\mathcal{D}_{-}\subset \mathcal{D}_{-},\: s\leq
0,$ and $U_{h}(s)\mathcal{D}_{+}\subset \mathcal{D}_{+},\: s\geq 0$;\\
\item [2)]  $\underset{s\leq 0}{\bigcap }U_{h}(s)\mathcal{D}_{-}=%
\underset{s\geq 0}{\bigcap }U_{h}(s)\mathcal{D}_{+}=\left\{ 0\right\}$;\\
\item [3)] $\overline{\underset{s\geq 0}{\bigcup }U_{h}(s)\mathcal{D}_{-}}=\overline{\underset{s\leq 0}{\bigcup }U_{h}(s)\mathcal{D}_{+}}=\mathcal{H}$;\\
\item [4)] $\mathcal{D}_{-}\bot \mathcal{D}_{+}.$
\end{itemize}
Since the proof of these items for $\mathcal{D}_{+}$ and for $\mathcal{D}_{-}$ are similar, we only prove them for $\mathcal{D}_{+}$. To prove item (1), we firstly set $R_{\lambda}(M_{h}-\lambda I)^{-1}$ for all $\lambda$ with $\operatorname{Im}\lambda<0$. Then, for any $f=\left \langle 0,0,\phi_{+}\right \rangle\in\mathcal{D}_{+}$, we write $$R_{\lambda}f=\left \langle 0,0,-ie^{i\lambda\xi}\int_{0
}^{\xi}e^{i\lambda s}{\phi _{+}}(s)ds\right \rangle.$$ It means that $R_{\lambda}f\in \mathcal{D}_{+}$. Therefore, if $g\bot\mathcal{D}_{+}$, then for $\operatorname{Im}\lambda<0$, we write $$(R_{\lambda}f,g)_{\mathcal{H}}=-i\int_{0}^{\infty}e^{-i\lambda s}{(U_{h}f,g)_{\mathcal{H}}}ds=0$$ and hence $(U_{h}f,g)_{\mathcal{H}}=0$ for all $s\geq0$. It follows from that $U_{h}(s)\mathcal{D}_{+}\subset \mathcal{D}_{+}$ for $s\geq0$ and it completes the proof of (1). 
To present the proof of (2), let us represent the mappings defined by $\mathcal{P}%
_{1}^{+}:\mathcal{H}%
\rightarrow \mathcal{L}^{2}\left( \mathbb{R}_{+}\right) $ and $\mathcal{P}%
_{2}^{+}:\mathcal{L}^{2}(\mathbb{R}_{+})\rightarrow \mathcal{D}_{+}$ as 
$\mathcal{P}_{1}^{+}:\left\langle
\phi_{-},y,\phi_{+}\right\rangle $ $\rightarrow \phi_{+}$ and $%
\mathcal{P}_{2}^{+}:\phi \rightarrow \left\langle 0,0,\phi\right\rangle,$ respectively. Note that the semigroup of isometries $U_{h}^{+}(s)=\mathcal{P}_{1}^{+}U_{h}(s)\mathcal{P}_{2}^{+},\:s\geq 0,$ is a one-sided shift in $\mathcal{L}^{2}\left( \mathbb{R}_{+}\right)$. Indeed, the generator of the semigroup corresponding to the one-sided shift $V_{h}$ in $\mathcal{L}^{2}\left( \mathbb{R}_{+}\right)$ is the differential operator $i\frac{d}{d\xi }$ with the boundary condition $\phi\left( 0\right) =0.$ On the other hand, the generator $S$ of the semigroup of isometries $U_{h}^{+}(s)$ in $\mathcal{L}^{2}\left( \mathbb{R}_{+}\right)$ is the operator $S\phi=\mathcal{P}_{1}^{+}M_{h}\mathcal{P}_{2}^{+}\phi$ $=\mathcal{P}_{1}^{+}M_{h}\left\langle 0,0,\phi\right\rangle $ $=\mathcal{P}_{1}^{+}\langle 0,0,i\frac{d\phi}{d\xi }\rangle$ $=i\frac{d\phi }{d\xi },$
where $\phi \in \mathcal{W}_{2}^{1}\left( \mathbb{R}_{+}\right) $ and $
\phi \left( 0\right) =0.$ Since a semigroup is uniquely determined by its generator, it gives that $U_{h}^{+}(s)=V_{h}(s)$, and hence 
\begin{equation*}
\underset{s\geq 0}{\bigcap }U_{h}^{+}(s)\mathcal{D}^{+}=\langle 0,0,%
\underset{s\geq 0}{\bigcap }V_{h}(s)\mathcal{L}^{2}\left( \mathbb{R}%
_{+}\right)\rangle =\{0\},
\end{equation*}%
that is, the proof of (2) is completed. The proof of item (4) is clear. It is also known from Lax-Phillips scattering theory that the scattering matrix is defined via the theory of spectral representations. In the process, we also establish item (3) concerning the incoming and outgoing subspaces. Firstly, we will prove the following theorem after giving the definition of the completely non-selfadjoint operator.
\begin{definition}
A linear operator $A$ with the domain $D(A)$ acting in a Hilbert space $H$ is called completely non-selfadjoint or simple if there is no invariant subspace $N\:(N\neq \{0\})$ of the operator $A$ such that $D(A)\supseteq N$. Here, the restriction of $A$ to $N$ is selfadjoint \cite{5,17,23}.
\end{definition}
\begin{theorem}\label{T3*}
The operator $T_{h}$ is completely non-selfadjoint (simple).
\end{theorem}
\begin{proof}
Assume that $T_{h}$ is not completely non-selfadjoint. Then there exists a nontrivial subspace of $\mathcal{L}_{q}^{2}(q^{\mathbb{Z}};\mathbb{C}^{2})$ in which $T_{h}$ has a selfadjoint part $T_{h}$$^{\thicksim}$ in it. Let us call this subspace by $H^{\thicksim }$. If $f\in D(T_{h}^{\thicksim})$, then $f$ becomes in $D({T_{h}^{\thicksim}}^{\ast})$, and
\begin{align*}
f_{2}(0)-h f_{1}(0)  &= 0 \\
\phantom{-} f_{2}(0)-\overline{h}f_{1}(0) &= 0, 
\end{align*} 
such that $f=\begin{pmatrix}
f_1 \\
f_2
\end{pmatrix}$. From this for the eigenvectors $f(t,\lambda)$
of the operator $T_{h}$ that lie in $H^{\thicksim}$ and are eigenvectors of $T_{h}$$^{\thicksim}$, we have $f_{1}(0,\lambda)=0$, from the boundary condition $f_{2}(0)-h f_{1}(0)= 0$, we get $f_{2}(0,\lambda)=0$, and then we have $f(t,\lambda) \equiv 0$ by using the uniqueness theorem of the Cauchy problem for the system $(l_{1}f)(t)=\lambda f(t),\: t\in q^{\mathbb{Z}}$. Since $m_{\infty}(\lambda)$ is a meromorphic function in $\C$, it can be concluded that the resolvent $\mathcal{R}(T_{h})$ of the operator $T_{h}$ is a compact operator and hence the spectrum of $T_{h}$$^{\thicksim}$ is purely discrete. Then by using the theorem giving for expansion in eigenfunctions of the selfadjoint operator $T_{h}$$^{\thicksim}$, we write $H^{\thicksim}=\{0\}$, it means that the operator $T_{h}$ is simple and it completes the proof of Theorem.\end{proof}
To write the proof of item (3), let us set 
\begin{equation*}
H_{-}=\overline{\underset{s\geq 0}{\bigcup }U_h(s)\mathcal{%
D}_{-}},\quad {H}_{+}=\overline{\underset{s\leq 0}{\bigcup }U_{h}%
(s)\mathcal{D}_{+}}.
\end{equation*}
\begin{lemma}\label{L7}
The equality $H_{-}+{H}_{+}=\mathcal{H}$ satisfies. 
\end{lemma}
\begin{proof}
 Considering item (1) of the subspace $\mathcal{D}_{+}$, it is easy for us to show that the subspace $\mathcal{H}^{\thicksim}=\mathcal{H})\oplus(H_{-}+{H}_{+})$ is invariant under the group $\{U_{h}(s)\}
$ and has the form $\mathcal{H}^{\thicksim}=\langle0,{H}^{\thicksim},0\rangle$, where ${H}^{\thicksim}$ is a subspace of $\mathcal{L}_{q}^{2}(q^{\mathbb{Z}};\mathbb{C}^{2})$. Therefore, if the subspace $\mathcal{H}^{\thicksim}$ and also ${H}^{\thicksim}$ were not trivial, the unitary group $\{U_{h}\}^{\thicksim}$, restricted to this subspace, would be a unitary part of the group $\{U_{h}\}$, and hence the restriction of $T_{h}$ to ${H}^{\thicksim}$ would be a selfadjoint operator in ${H}^{\thicksim}$. Since $T_{\alpha}$ is simple, it gives ${H}^{\thicksim}=\{0\}$, i.e., $\mathcal{H}^{\thicksim}=\{0\}$. It completes the proof of Lemma.
\end{proof}
Let us introduce the following notations to present next results:
$w(x,\lambda):=\psi(x,\lambda)+m_{\infty}(\lambda)\phi(x,\lambda)$,
\begin{equation}\label{3.2}
S _{h}(\lambda ):=\frac{m_{\infty}(\lambda )-h}{m_{\infty}(\lambda )-\overline{h}}. 
\end{equation}
Now, consider the vectors $U_{\lambda}^{\pm}(x,\xi,\tau)$ as
\begin{equation}\label{3.3}
 U_{\lambda}^{-}(x,\xi,\tau)=\langle \exp\{-i\lambda\xi\},(m_{\infty}(\lambda)-h)^{-1}\delta w(x,\lambda),\overline{S_{h}}(\lambda)exp\{-i\lambda\tau\}\rangle
\end{equation}
and
\begin{equation}\label{3.4}
 U_{\lambda}^{+}(x,\xi,\tau)=\langle S_{h}(\lambda)\exp\{-i\lambda\xi\},(m_{\infty}(\lambda)-\overline{h})^{-1}\delta w(x,\lambda),exp\{-i\lambda\tau\}\rangle.
\end{equation}
Note that these vectors do not belong to the space $\mathcal{H}$ for real $\lambda$. However, these vectors satisfy the equation $MU=\lambda U$ and the boundary conditions for $M_{h}$. We define the following transformations $F_{\pm}:f\to \tilde f_{\pm}(\lambda)$ using the vectors $U_{\lambda}^{\pm}(x,\xi,\tau)$ for the next steps as follows
\begin{equation*}
(F_{\pm}f)=\tilde f_{\pm}(\lambda)=\frac{1}{\sqrt {2\pi}}(f,U_{\lambda}^{\pm})_{\mathcal{H}},
\end{equation*}
on the vector $\langle\phi_{-},y,\phi_{+}\rangle$, where $\phi_{+},y,\phi_{-},$ are compactly supported smooth functions. Assume $f=\langle\phi_{-},0,0\rangle$ and $g=\langle\psi_{-},0,0\rangle$ in $D_{-}$. Then, the equality 
\begin{equation*}
\tilde f_{-}(\lambda)=\frac{1}{\sqrt {2\pi}}(f,U_{\lambda}^{-})_{\mathcal{H}}=\frac{1}{\sqrt {2\pi}}\int_{-\infty}^{0}\phi_{-}(s)\exp{(i\lambda s)}ds
\end{equation*} 
satisfies. Hence $\tilde f_{-}(\lambda)$ belongs to $\mathcal{H}^{2}_{-}$. Now, consider the dense set $\tilde H_{-}$ in $H_{-}$ consisting of all vectors $f$ such that $f$ is a compactly supported function in $D_{-}$ and $f=U_{h}(s)f_{0}$, $f_{0}=\langle\phi_{-},0,0\rangle$, $\phi_{-}\in C^{\infty}_{0}(\mathbb{R}_{-})$, where $C^{\infty}_{0}(\mathbb{R}_{-})$ is the set of all smooth compactly supported functions defined in $\mathbb{R}_{-}$ and  $s=sf$ is a nonnegative number. In this case, if $f,g\in H_{-}$, we have for $s>s_{f}$ and $s>s_{g}$ that $U_{h}(-s)f$, $U_{h}(-s)g\in D_{-}$ and their first components belong to $C^{\infty}_{0}(\mathbb{R}_{-})$. Since the operators $U_{h}$ $(-\infty<h<\infty)$ are unitary, by using the equality
\begin{equation*}
F_{-}U_{h}f=(U_{h}f,U^{-}_{\lambda})_{\mathcal{H}}=e^{i\lambda h}(f, U^{-}_{\lambda})_{\mathcal{H}} = e^{i\lambda h}F_{-}f,
\end{equation*} 
we write
\begin {equation}\label{3.5}
\begin{aligned}
(f,g)_{\mathcal{H}} &= (U_h(-s)f, U_h(-s)g)_{\mathcal{H}} 
= (F_{-}U_h(-s)f, F_{-} U_h(-s)g)_{\mathcal{L}^{2}} \\
&= (e^{-i\lambda s}F_{-}f, e^{-i\lambda s}F_{-}g)_{\mathcal{L}^{2}} 
= (F_{-}f, F_{-}g)_{\mathcal{L}^{2}}.
\end{aligned}
\end{equation}
By taking closure in \eqref{3.5}, we get Parseval equality for the space $H_{-}$. Then the inversion formula $f=\frac{1}{\sqrt {2\pi}}\int_{-\infty}^{\infty}\tilde f_{-}(\lambda)\overline{U^{-}_{\lambda}}d\lambda$
follows from the Parseval equality if all integrals are taken as limits in the mean of the intervals. Finally, we obtain
\begin{equation*}
F_{-}H_{-}=\overline{\underset{s\geq 0}{\bigcup }F_{-}U_h(s)\mathcal{%
D}_{-}}=\overline{\underset{s\geq 0}{\bigcup }e^{-i\lambda s}\mathcal{H}_{-}^{2}}={\mathcal{L}^{2}}(\mathbb{R}).
\end{equation*}
Last equality implies that $H_{-}$ is isometrically identical with ${\mathcal{L}^{2}}(\mathbb{R})$ and for all vectors $f,g\in H_{+}$, the Parseval equality and the inversion formula can be written as
\begin{equation*}
(f,g)_{\mathcal{H}}=(\tilde f_{+},\tilde g_{+})_{\mathcal{L}^{2}}=\int_{-\infty}^{\infty}\overline{f_{+}}(\lambda)\overline{\tilde g_{+}(\lambda)}d\lambda,\qquad f=\frac{1}{\sqrt {2\pi}}\int_{-\infty}^{\infty}\tilde f_{+}(\lambda)U^{+}_{\lambda}d\lambda
\end{equation*}
where $\tilde f_{+}(\lambda)=(F_{+}f)(\lambda)$ and $\tilde g_{+}(\lambda)=(F_{+}g)(\lambda)$. Furthermore, the function $S_{h}(\lambda)$ satisfies $\left| S_{h}(\lambda) \right|=1$ fpr $\lambda\in \mathbb{R}$ according to \eqref{3.2}. Hence the explicit expressions for the vectors $U^{+}_{\lambda}$ and $U^{-}_{\lambda}$ that 
\begin{equation}\label{3.6}
U^{-}_{\lambda}= \overline{S_{h}}U^{+}_{\lambda},\quad \lambda\in\mathbb{R}.
\end{equation}
It follows from these expressions that $H_{-}=H_{+}$ and if we consider this with Lemma 7, we obtain $\mathcal{H}=H_{-}=H{+}$, and property (3) has been constructed for the incoming and outgoing subspaces.
Thus, the transformation $F_{-}$ is defined as isometric mapping from 
$\mathcal{L}^{2}(\mathbb{R})$ to subspace $D_{-}$, which is carried out in $H^{2}_{-}$, while the operators $U_{h}$ transform into multiplication operators by $S$. In other words, $F_{-}$ is the incoming spectral representation and $F_{+}$ is the outgoing spectral representation for the group ${U_{h}}$. It follows from \eqref{3.6} that the passage from the $F_{+}$-representation of a vector $f\in \mathcal{H}$ to its $F_{-}$-representation is realized by multiplication of the function $S_{h}(\lambda):=\tilde f_{-}(\lambda)=S_{h}(\lambda)\tilde f_{+}(\lambda)$. As stated in \cite{18}, the scattering function (matrix) of the group $\{U_{h}\}$ with respect to the spaces $D_{-}$ and $D_{+}$ is the factor by which the $F_{-}$-representation of a vector $f\in\mathcal{H}$ must be scaled to obtain its corresponding $F_{+}$-representation $\tilde f_{+}(\lambda)=\overline{S_{h}}(\lambda)\tilde f_{-}(\lambda)$, and thus we present the following statement.
\begin{theorem}\label{T4}
The function $\overline{S_{h}}$ is the scattering  matrix of the group $\{U_{h}\}$, i.e.; of the selfadjoint operator $M_{h}$.
\end{theorem}
To present next Theorem, let us give the following information.
Let $S(\lambda)$ be an arbitrary inner function \cite{20} on the upper half plane. Define $\mathcal{K}=H^{2}_{+}\oplus SH^{2}_{+}$. Then $\mathcal{K}\neq\{0\}$ is a subspace of the Hilbert space $H^{2}_{+}$. We consider the semigroup of the operators $Z_{h}\:(h\geq0)$ acting in $\mathcal{K}$ according to the formula $Z_{h}\phi=P_{1}[e^{i\lambda h}\phi]$, $\phi:=\phi(\lambda)\in\mathcal{K}$, where $P_{1}$ is the orthogonal projection from $H^{2}_{+}$ onto $\mathcal{K}$. The operator of the semigroup $Z_{h}$ is denoted by
$$L:L\phi=\lim_{h\to+0}(ih)^{-1}(Z_{h}\phi-\phi),$$ which is a dissipative operator that acts in $\mathcal{K}$ and with the domain $D(L)$ consisting of all functions $\phi\in\mathcal{K}$, such that the limit exists. The operator $L$ is called a model dissipative operator (this model dissipative operator, which is associated with the names of Lax and Phillips \cite{18}, is a special case of a more general model dissipative operator established by Sz. Nagy and Foia\c{s} \cite{20}). The basic report is that $S(\lambda)$ is the characteristic function of the operator $L$.
Let $\textbf{K}=\langle 0,\mathcal{L}_{q}^{2}(q^{\mathbb{Z}};\mathbb{C}^{2}),0\rangle$, so that we write $\mathcal{H}=\mathcal{D}_{-}\oplus\textbf{K}\oplus\mathcal{D}_{+}$. It follows from the explicit form of the unitary transformation $F_{-}$ that under the mapping $F_{-}$
\begin{equation}\label{3.7}
\mathcal{H}\to \mathcal{L}^{2}(\mathbb{R}),\quad f\to\tilde f_{-}(\lambda)=(F_{-}f)(\lambda),\quad \mathcal{D}_{-}\to H^{2}_{-},\quad \mathcal{D}_{+}\to S_{h}H^{2}_{+}
\end{equation}
and
\begin{equation}\label{3.8}
\textbf{K}\to H^{2}_{+}\oplus S_{h}H^{2}_{+},\quad U_{h}(s)f\to (F_{-}U_{h}(s)F^{-1}_{-}\tilde f_{-})(\lambda)=e^{i\lambda s}\tilde f_{-})(\lambda).
\end{equation}
The formulas \eqref{3.7} and  \eqref{3.8} show that our operator $T_{h}$ is a unitary equivalent to the model dissipative operator with characteristic function $S_{h}(\lambda)$. It is known from the literature that the characteristic functions of unitary equivalent dissipative operators coincide (\cite{20}), so we can present the following theorem.
\begin{theorem}\label{T5}
 The characteristic function of the dissipative operator $T_{h}$ coincides with the function $S_{h}(\lambda)$ defined in \eqref{3.2} \end{theorem}
\section{Completeness theorems for the system of eigenvectors and associated vectors of the dissipative operators}
In this section, we prove theorems on the completeness of the system of eigenvectors and associated vectors of dissipative $q$-Dirac operators by using the theory of characteristic function. It is known that the characteristic function of a dissipative operator contains complete information about the spectral properties of the operator \cite{17,20}. For example, if the characteristic function $S_{h}(\lambda)$ admits a factorization of the form $S_{h}(\lambda)=s(\lambda)\mathbb{B}_{h}(\lambda)$ and the singular factor $s(\lambda)$ is absent, this ensures that the system of eigenvectors and associated vectors of the dissipative $q$-Dirac operator $T_{h}$, where $\mathbb{B}(\lambda)$ is a Blaschke product.
\begin{theorem}\label{T6}
For all $h$ such that the imaginary part of $h$ is positive, except possibly for a single value $h=h_{0}$, the characteristic function $S_{h}(\lambda)$ of the dissipative operator $T_{h}$ is a Blaschke product, and the spectrum of  $T_{h}$ is purely discrete and belongs to the open upper half-plane. The operator $T_{h}$ $(h\neq h_{0})$ has a countable number of isolated eigenvalues with finite multiplicities and limit points at infinity, and the system of eigenvector and associated vectors of this operator is complete in $\mathcal{L}_{q}^{2}(q^{\mathbb{Z}};\mathbb{C}^{2})$.
\end{theorem}
\begin{proof}
It is written from the explicit formula \eqref{3.2} that $S_{h}(\lambda)$ is an inner function in the upper half-plane, and moreover, it is a meromorphic in the whole complex $\lambda$-plane. Therefore it can be factored as follows
\begin{equation}\label{4.1}
S_{h}(\lambda)=e^{i\lambda d}\mathbb{B}_{h}(\lambda),\quad d=d(h)\geq 0, 
\end{equation}
where $\mathbb{B}_{h}(\lambda)$ is a Blaschke product. It follows from \eqref{4.1} that 
\begin{equation}\label{4.2}
\left| S_{h}(\lambda) \right|=\left| e^{i\lambda d} \right|\left| \mathbb{B}_{h}(\lambda) \right|\leq e^{-d(h)\operatorname{Im} \lambda},\quad \operatorname{Im} \lambda\geq 0.
\end{equation}
Then expressing $m_{\infty}(\lambda)$ in terms of $S_{h}(\lambda)$, we find
\begin{equation}\label{4.3}
m_{\infty}(\lambda)=\frac{\overline{h}S_{h}(\lambda)-h}{S_{h}(\lambda)-1}    
\end{equation}
by using \eqref{3.2}. If $d(h)>0$ for a given value $h$ with $\operatorname{Im}\lambda>0$, then by \eqref{4.3}, we obtain $\lim_{s\to\infty}S_{h}(is)=0$, which together with \eqref{4.3} implies that $\lim_{s\to\infty}m_{\infty}(is)=-h$. Since $m_{\infty}(\lambda)$ is independent of $h$, $d(h)$ can be nonzero at not more than a single point $h=h_{0}$ and, further, $h_{0}=-\lim_{s\to\infty}m_{\infty}(is)$. It completes the proof.
\end{proof}
Note that all results which are obtained for dissipative operators can also be represented for accumulative operators, because a linear operator $A$ acting in a Hilbert space $H$ is accumulative if and only if $-A$ is dissipative. Then by using the previous theorem, we can directly write the following result.
\begin{remark}
Let the function $m_{\infty}$ be meromorphic in $\mathbb{C}$. Then, for all values of $h$ with $\operatorname{Im}h<0$, except possibly for a single value $h=h_{1}$, the spectrum of accumulative operator $T_{h}$ is purely discrete and belongs to the open lower half-plane. The operator $T_{h}\quad (h\neq h_{1})$ has a countable number of isolated eigenvalues with finite multiplicities and limit points at infinity, and the system of all eigenvectors and associated vectors of the operator $T_{h}$ is complete in $\mathcal{L}_{q}^{2}(q^{\mathbb{Z}};\mathbb{C}^{2})$.
\end{remark}

\textbf{Acknowledgements}  The authors would like to thank the reviewers for their reviews.

\textbf{Data Availability} Our manuscript does not have associated data.
\section*{Declarations}
\textbf{Conflict of interest} The authors declare that they have no conflict of interest.

\end{document}